\documentclass[english,11pt]{amsart}
\usepackage[english]{babel}

\usepackage[matrix,arrow]{xy}
\xyoption{all}
\usepackage{amscd,amssymb,amsfonts,amsmath}

\usepackage{graphics}
\usepackage{epsfig}
\usepackage{mathrsfs}

\newcommand\mypagesizel{
\textwidth= 6.5in
\textheight=9in
\voffset-.55in
\hoffset -0.75in
\marginparwidth=56pt
}

\renewcommand{\phi}{\varphi}
\newcommand{\into}{\hookrightarrow}

\renewcommand{\le}{\leqslant}
\renewcommand{\ge}{\geqslant}

\newcommand{\bC}{\textup{\textbf{C}}}

\newcommand{\bN}{\textbf{N}}

\newcommand{\bP}{\textup{\textbf{P}}}

\newcommand{\bZ}{\textup{\textbf{Z}}}

\newcommand{\cH}{\mathcal{H}}

\newcommand{\frS}{\mathfrak{S}}

\newcommand{\sE}{\mathscr{E}}
\newcommand{\sF}{\mathscr{F}}
\newcommand{\sG}{\mathscr{G}}

\newcommand{\sL}{\mathscr{L}}
\newcommand{\sM}{\mathscr{M}}

\newcommand{\sO}{\mathscr{O}}

\newcommand{\sQ}{\mathscr{Q}}

\newcommand{\codim}{\textup{codim}}

\mypagesizel

\newtheorem{thm}{Theorem}
\newtheorem*{thma}{Theorem A}
\newtheorem*{thmb}{Theorem B}

\newtheorem{lemma}[thm]{Lemma}
\newtheorem{cor}[thm]{Corollary}

\newtheorem{prop}[thm]{Proposition}

\newtheorem*{thm*}{Theorem}

\theoremstyle{definition}

\newtheorem{notation}[thm]{Notation}

\newtheorem{defn-thm}[thm]{Definition-Theorem}

\theoremstyle{remark}

\newtheorem*{acknowledgments}{Acknowledgments}

\newtheorem*{not-and-def}{Notation and definitions}

\begin{document}

\title[Characterizations of projective spaces and hyperquadrics]{
Characterizations of projective spaces and hyperquadrics}

\author{St\'ephane Druel}

\author{Matthieu Paris}

\address{Institut Fourier, UMR 5582 du
  CNRS, Universit\'e Grenoble 1, BP 74, 38402 Saint Martin
  d'H\`eres, France} 

\email{druel@ujf-grenoble.fr}

\email{Matthieu.Paris@ujf-grenoble.fr}

\thanks{The first named author was partially supported by the A.N.R}


\maketitle

\section{Introduction}

Projective spaces and hyperquadrics are the simplest projective
algebraic varieties, and they can be characterized in many ways.  The
aim of this paper is to provide a new characterization of them in
terms of positivity properties of the tangent bundle. 
We refer the reader to the article \cite{adk08} which reviews these matters. Notice that
our results generalize
Mori's (see \cite{mori79}), Wahl's (see \cite{wahl} and \cite{druel04}),
Andreatta-Wi\'sniewski's (see \cite{andreatta_wisniewski} and \cite{araujo06}),
Araujo-Druel-Kov\'acs's (see \cite{adk08}) and Paris's (see \cite{paris})
characterizations of projective spaces and hyperquadrics.
K. Ross recently posted
a somewhat related result (see \cite{ross}).

\medskip

In this paper, we prove the following theorems. Here $Q_n$ denotes a smooth
quadric hypersurface in $\bP^{n+1}$, and $\sO_{Q_{n}}(1)$ denotes the 
restriction of $\sO_{\bP^{n+1}}(1)$ to $Q_n$.
When $n=1$, $(Q_{1},\sO_{Q_{1}}(1))$ is just $(\bP^{1},\sO_{\bP^{1}}(2))$.

\begin{thma}
Let $X$ be a smooth complex projective $n$-dimensional variety 
and $\sE$ an ample vector bundle on $X$ of rank $r+k$ with $r\ge 1$ and $k\ge 1$.
If 
$h^0(X,T_{X}^{\otimes r}\otimes \det(\sE)^{\otimes -1})\neq 0$, then
$(X,\det(\sE))\simeq (\bP^{n},\sO_{\bP^{n}}(l))$ with
$r+k\le l\le \frac{r(n+1)}{n}$.
\end{thma}

\begin{thmb}
Let $X$ be a smooth complex projective $n$-dimensional variety and
$\sE$ an ample vector bundle on $X$ of rank $r\ge 1$.  If 
$h^0(X,T_{X}^{\otimes r}\otimes \det(\sE)^{\otimes -1})\neq 0$, then
either $(X,\det(\sE))\simeq (\bP^{n},\sO_{\bP^{n}}(l))$ with $r \le l\le \frac{r(n+1)}{n}$,
or $(X,\sE)\simeq(Q_{n},\sO_{Q_{n}}(1)^{\oplus r})$ and $r=2i+nj$ with $i \ge 0$ and $j\ge 0$.
\end{thmb}

The line of argumentation follows \cite{andreatta_wisniewski} (see also \cite{adk08} and \cite{paris}). 
We first prove Theorem A and Theorem B for Fano manifolds with Picard number $\rho(X)=1$ (see Proposition
\ref{proposition:picard_number_one}).
Then the argument for the proof of the main
Theorem goes as follows. We argue by induction on $\dim(X)$.
We may assume $\rho(X) \ge 2$. Hence the
$H$-rationally connected quotient of $X$ with respect to an unsplit covering 
family $H$ of rational curves on $X$ is non-trivial. It can be extended in codimension one so that 
we can produce a normal
variety $X_B$ equipped with a surjective morphism $\pi_B$ with integral fibers onto a smooth curve $B$
such that either
$B\simeq \bP^1$,
$X_B\to B$ is a $\bP^d$-bundle for some $d\ge 1$
and 
$h^0(X_B,T_{X_B/\bP^1}^{\otimes i}\otimes \pi^*\sG^{\otimes r-i}\otimes
\det(\sE)_{|X_B}^{\otimes -1})\neq 0$ for some integer $1\le i\le r$
where $\sG$ be a vector bundle on $\bP^1$ such that $\sG^*(2)$ is nef, or
$X_B\to B$ is a $\bP^d$-bundle for some $d\ge 1$ and 
$h^0(X_B,T_{X_B/B}^{\otimes r}\otimes \det(\sE)_{|X_B}^{\otimes -1}\otimes \pi_B^*\sG^*)\neq 0$
where $\sG$ is a nef vector bundle on $C$, or
the geometric generic fiber of $\pi_B$ is isomorphic to a smooth hyperquadric and 
$h^0(X_B,T_{X_B/B}^{[\otimes r]}\otimes \det(\sE)_{|X_B}^{\otimes -1}\otimes \pi_B^*\sG^*)\neq 0$
where $\sG$ is a nef vector bundle on $C$.
But this is impossible unless $X\simeq\bP^1\times \bP^1$ (see Lemma \ref{lemma:bundle_over_line}, 
Lemma \ref{lemma:fibration_over_curve_projective_space}
and Proposition \ref{proposition:fibration_over_curve_quadric}).
\medskip

Throughout this paper we work over the field of complex numbers. 

\begin{acknowledgments}We are grateful to Nicolas \textsc{Perrin} for very fruitful discussions.
\end{acknowledgments}

\section{Proofs}

\subsection{Projective spaces and hyperquadrics}In this section, we gather some properties of the tangent bundle to projective spaces and smooth hyperquadrics.

\begin{lemma}\label{lemma:twisted_sections_projective_space}
Let $n$, $r$ and $k$ be integers with $n \ge 1$ and $r \ge 1$. Then
$h^0(\bP^n,T_{\bP^n}^{\otimes r}(-k))\neq 0$ if and only if $k \le \frac{r(n+1)}{n}$. 
\end{lemma}

\begin{proof}It is well-known that $T_{\bP^n}$ is stable in the sense of 
Mumford-Takemoto with slope $\mu(T_{\bP^n})=\frac{n+1}{n}$ with respect to
$\sO_{\bP^n}(1)$. 
By \cite[Theorem 3.1.4]{HuyLehn}, $T_{\bP^n}^{\otimes r}(-r)$ is semistable with slope
$\mu(T_{\bP^n}^{\otimes r}(-k))=\frac{r(n+1)}{n}-k$.
It follows that if $h^0(\bP^n,T_{\bP^n}^{\otimes r}(-k))\neq 0$
then $k \le \frac{r(n+1)}{n}$. Conversely, let us assume that $k \le \frac{r(n+1)}{n}$.
Write $r=an+b$ where $a$ and $b$ are integers with $a \ge 0$ and
$0\le b<n$. Then 
$k-a(n+1)=\llcorner k-a(n+1)\lrcorner
\le \llcorner \frac{b(n+1)}{n} \lrcorner 
=\llcorner b+\frac{b}{n}\lrcorner=b$ and
\begin{eqnarray*}
h^0(\bP^n,T_{\bP^n}^{\otimes r}(-k) & = &
h^0(\bP^n,T_{\bP^n}^{\otimes an}(-a(n+1))\otimes T_{\bP^n}^{\otimes b}(-k+a(n+1)))\\
& \ge & h^0(\bP^n,[T_{\bP^n}^{\otimes n}(-(n+1))]^{\otimes a}\otimes T_{\bP^n}^{\otimes b}(-b))\\
& \ge & h^0(\bP^n,[\det(T_{\bP^n})(-(n+1))]^{\otimes a}\otimes T_{\bP^n}^{\otimes b}(-b))\\
& = & h^0(\bP^n,T_{\bP^n}^{\otimes b}(-b)) \ge 1,
\end{eqnarray*}
as claimed.
\end{proof}

Let $d$ be a positive integer. Let $Q\subset \bP^{d+1}=\bP(W)$ be a smooth hyperquadric
defined by a non degenerate quadratic form $q$ on $W:=\bC^{d+2}$
and let $\sO_Q(1)$ denote the 
restriction of $\sO_{\bP^{d+1}}(1)$ to $Q$.
Let $x$ be a point of $Q$ and $w\in W\setminus\{0\}$ representing $x$; then $T_Q(-1)_x$
identifies with $x^{\perp}/<x>$ and $q$ induces an isomorphism 
$T_Q(-1)\simeq \Omega_Q^1(1)$ or equivalently a nonzero section in
$H^0(Q,(T_Q(-1))^{\otimes 2})$ still denoted by $q$. Let $V:=x^{\perp}/<x>$.
Let $G:=SO(W)$ and let $P\subset SO(W)$ be the parabolic subgroup such that
$G/P\simeq Q$ corresponding to $x\in Q$. Let $\det\in H^0(Q,\det(T_Q(-1))$ be a nonzero section.

\begin{lemma}\label{lemma:quadric_twisted_sections}
Let the notations be as above.
\begin{enumerate}
\item The vector bundle $T_Q$ is stable in the sense of Mumford-Takemoto; in particular, one has
$h^0(Q,T_Q^{\otimes r}(-k))=0$ for $k>r\ge 1$.
\item The space of sections $H^0(Q,(T_Q(-1))^{\otimes r})$ is generated as a $\bC$-vector space 
by the $\sigma\cdot q^{\otimes i}\otimes \det ^{\otimes j}$'s where $i$ and $j$ are nonnegative integers such that
$r=2i+dj$ and $\sigma \in
\frS_r$ the symmetric group on $r$ letters acting as usual on the vector bundle $(T_Q(-1))^{\otimes r}$.
\end{enumerate}
\end{lemma}

\begin{proof}Observe that $T_Q(-1)$ is homogeneous or equivalently that 
$$T_Q(-1)\simeq (G \times V)/P$$ over $Q\simeq G/P$ where $g \in P$ acts on $G \times V$ by the formula
$$g\cdot (g',v)=(g'g,\rho(g^{-1})\cdot v)$$ and
$$\rho: P \to GL(T_Q(-1)_x)=GL(V)$$ is the stabilizer representation. It vanishes on the unipotent radical 
$U$ of $P$ and can be viewed as the representation of the Levi subgoup 
$L\simeq \bC^*\times SO(V)\subset P$ on $V$ given
by the standard representation of $SO(V)$ on $V$. It is irreducible and therefore $T_Q(-1)$ is indecomposable hence stable by \cite{ramanan66} and \cite{umemura78} with slope 
$\mu(T_Q(-1))=0$ with respect to $\sO_Q(1)$. 
By \cite[Theorem 3.1.4]{HuyLehn}, $(T_Q(-1))^{\otimes k}$ is semistable with slope
$\mu((T_Q(-1))^{\otimes r})=0$. This ends the proof of the first part of the Lemma.

Observe that $(T_Q(-1))^{\otimes r}$ is homogeneous
and that the stabilizer representation $$P \to GL((T_Q(-1))^{\otimes r}_x)$$ is 
$\rho^{\otimes r}$. In particular, $(T_Q(-1))^{\otimes r}$ decomposes as the direct sum of 
indecomposable vector bundles hence as the direct sum of stable vector bundles with slope $0$.
It follows that there is a one-to-one correspondence between the set of nonzero section in $H^0(Q,(T_Q(-1))^{\otimes r})$ and the set of rank one direct summands of 
$T_Q(-1))^{\otimes r}$. Finally, we obtain an isomorphism 
$$H^0(Q,(T_Q(-1))^{\otimes r})\simeq (V^{\otimes r})^{SO(V)}$$ since $SO(V)$ has no nontrivial character. The result now
follows from \cite[Theorem 2.9 A]{weyl}.
\end{proof}

\subsection{Fibrations over curves}In this section, we prove our main Theorems for fibrations over curves.

\begin{lemma}\label{lemma:bundle_over_line}
Let $\sF$ be a vector bundle on $\bP^1$ of rank $m\ge 2$, $X:=\bP_{\bP^1}(\sF)$
and $\pi:X\to\bP^1$ the natural morphism.
Let $\sE$ be an ample vector bundle on $X$ of rank $r+k$ with $r\ge 2$ and $k\ge 0$.
Let $\sG$ be a vector bundle on $\bP^1$ such that $\sG^*(2)$ is nef.
If $h^0(X,T_{X/\bP^1}^{\otimes i}\otimes \pi^*\sG^{\otimes r-i}\otimes
\det(\sE)^{\otimes -1})\neq 0$ for some integer $0\le i\le r$
then $X\simeq \bP^1\times\bP^1$, $\sF=\sO_{\bP^1}(a)^{\oplus 2}$ for some integer $a$, $k=0$, $2i=r$
and $\det(\sE)\simeq\sO_{\bP^1}(2)\boxtimes\sO_{\bP^1}(2)$.
\end{lemma}

\begin{proof}Write $\sF\simeq\sO_{\bP^1}(a_1)\oplus\cdots\oplus\sO_{\bP^1}(a_m)$
with $a_1\le\cdots \le a_m$. Let $b:=a_m-a_1\ge 0$. Let $\sigma$ be a section of $\pi$ corresponding to a
surjective morphism 
$\sO_{\bP^1}(a_1)\oplus\cdots\oplus\sO_{\bP^1}(a_m)
\twoheadrightarrow\sO_{\bP^1}(a_m)$ and let
$\sigma_{1}$ the section of $\pi$ corresponding to the projection map 
$\sO_{\bP^1}(a_1)\oplus\cdots\oplus\sO_{\bP^1}(a_m)
\twoheadrightarrow\sO_{\bP^1}(a_1)$.
Then $\sigma\equiv \sigma_1+b\ell$ where $\ell$ is vertical line and
$$\det(\sE)\cdot\sigma \ge r+k +b(r+k)=(r+k)(b+1).$$
We may assume that 
$h^0(\sigma,(T_{X/\bP^1}^{\otimes i}\otimes \pi^*\sG^{\otimes r-i} \det(\sE)^{\otimes -1})_{|\sigma})\neq 0$
since $\sigma$ is a free rational curve. But
$${T_{X/\bP^1}}_{|\sigma}\simeq N_{\sigma/X}
\simeq\sO_{\bP^1}(a_m-a_1)\oplus\cdots\oplus\sO_{\bP^1}(a_m-a_{m-1})$$
and we obtain
\begin{equation}
(r+k)(b+1)\le \det(\sE)\cdot\sigma \le ib+2(r-i).
\end{equation}
By Lemma \ref{lemma:twisted_sections_projective_space}, we must have
\begin{equation}
r+k \le i\frac{m}{m-1}.
\end{equation}
We obtain
\begin{equation}
(r+k)(b+1)+2k \le ib+2(r+k)-2i \le ib+2i\frac{m}{m-1}-2i =i (b+\frac{2}{m-1}).
\end{equation}
It follows that $m=2$, $b=k=0$ and $r=2i$. 
\end{proof}

\begin{lemma}\label{lemma:fibration_over_curve_projective_space}
Let $X$ be a smooth complex projective variety,
$\sE$ an ample vector bundle on $X$ of rank $r+k$ with $r\ge 1$ and $k\ge 0$.
Let $\pi: X\to B$ be a surjective morphism onto a smooth connected curve with 
integral fibers. Let $\sG$ be a numerically effective
vector bundle on $B$ of rank $>0$.
Assume that the geometric generic fiber is isomorphic to a projective space.
Then $h^0(X,T_{X/B}^{\otimes r}\otimes \det(\sE)^{\otimes -1}\otimes \pi^*\sG^*)=0$.
\end{lemma}

\begin{proof}Let $\eta$ be the generic point of $B$. Tsen's Theorem implies that $X_{\eta}\simeq \bP^d_{k(\eta)}$.
Thus
there exists a divisor $H$ on $X$ such that
${\sO_X(H)}_{|X_\eta}\simeq \sO_{\bP^d_{k(\eta)}}(1)$. Let $\sL:=\sO_X(H)$. 
Let $r'\ge r+k$ be defined by the formula $\det(\sE)_{|X_\eta}\simeq\sO_{\bP^d_{k(\eta)}}(r')$.
It follows from the semicontinuity Theorem that 
$h^0(X_b,{(\det(\sE)\otimes\sL^{\otimes -r'})}_{|X_b})\ge 1$ and 
$h^0(X_b,{(\sL^{\otimes r'}\otimes\det(\sE)^{\otimes -1})}_{|X_b})\ge 1$ for any point $b$ in $B$.
Thus $h^0(X_b,{(\det(\sE)\otimes\sL^{\otimes -r'})}_{|X_b})=1$ since $X_b$ is integral. By the
base change Theorem, $\det(\sE)\simeq \sL^{\otimes r'}\otimes \pi^*\sM$ for some line bundle
$\sM$ on $B$. Thus $\sL$ is ample$/B$ and by \cite[Corollary 5.4]{fujita75}, $\pi$ is a $\bP^d$-bundle.
By replacing $B$ with a finite cover $\bar B\to B$ and $X$ with $X \times_B \bar B$ we may assume that
$g(B)\ge 1$. Let $\sM'$ be a line bundle on $B$ such that $\sM\simeq \sM'^{\otimes r'}$.
Set $\sL':=\sL\otimes\pi^*\sM'^{\otimes -1}$. Then $\sL'^{\otimes r'}\simeq \det(\sE)$ hence $\sL'$ is ample. 
Let 
$\sF:=\pi_*(\sL')$. Then $\sF$ is an ample vector bundle on $B$ and 
$X:=\bP_B(\sF)$. By \cite{campana_flenner90}, By replacing $B$ with a finite cover $\bar B \to B$ and $X$ with $X
\times_B \bar B$, we may assume that there exist an ample line bundle $\sM$ on $B$, a positive integer $m$ and a
surjective map of $\sO_B$-modules $\sM^{\oplus m}\twoheadrightarrow \sF$. Observe that the line bundle
$\sL'\otimes \pi^*\sM^{\otimes -1}$ is generated by its global sections. Let 
$C=D_1\cap \cdots \cap D_{\dim(X)-1}$
be general complete intersection
curve with $D_i \in |\sL'\otimes \pi^*\sM^{\otimes -1}|$
($C$ is a section of $\pi$). Then 
$(T_{X/B})_{|C} \simeq N_{C/X}\simeq (\sL'\otimes \pi^*\sM^{\otimes -1})_{|C}^{\oplus \dim(X)-1}$.
But 
$$
h^0(C,(\sL'\otimes \pi^*\sM^{\otimes -1})^{\otimes r}\otimes \det(\sE)^{\otimes -1}_{|C})\\
=h^0(C,\sL'^{\otimes r-r'} \otimes \pi^*\sM^{\otimes -r}_{|C})=0
$$
and the claim follows.
\end{proof}

When dealing with sheaves that are not necessarily locally free, we use square brackets to indicate taking the reflexive hull. 

\begin{notation}[Reflexive tensor operations]
Let $X$ be a normal variety and $\sQ$ a coherent sheaf of $\sO_X$-modules. 
For $n\in \bN$, set $\sQ^{[\otimes n]}:=(\sQ^{\otimes n})^{**}$, 
$S^{[n]}\sQ:=(S^n\sQ)^{**}$ and $\displaystyle{\det(\sQ):=(\wedge^{\textup{rank}(\sQ)}(\sQ))^{**}}$.
\end{notation}

\begin{prop}\label{proposition:fibration_over_curve_quadric}
Let $X$ be a normal complex projective variety,
$\sE$ an ample vector bundle on $X$ of rank $r+k$ with $r\ge 1$ and $k\ge 0$.
Let $\pi: X\to B$ be a surjective morphism onto a smooth connected curve with 
integral fibers. Let $\sG$ be a numerically effective
vector bundle on $B$ of rank $>0$.
Assume that the geometric generic fiber is isomorphic to a smooth hyperquadric.
Then $h^0(X,T_{X/B}^{[\otimes r]}\otimes \det(\sE)^{\otimes -1}\otimes \pi^*\sG^*)=0$.
\end{prop}

\begin{proof}Let $\eta$ be the generic point of $B$ and $k(\bar\eta)$ be an algebraic closure of $k(\eta)$.
Let $q_{\bar\eta}$ be a non degenerate quadratic form defining
$X_{\bar\eta}\subset \bP^{d+1}_{k(\bar\eta)}$ where $d:=\dim(X)-1$.
By Lemma \ref{lemma:quadric_twisted_sections},
$k=0$ and ${\det(\sE)}_{|X_{\bar\eta}}\simeq \sO_{X_{\bar\eta}}(r)$.

Let us assume to the contrary that 
$h^0(X,T_{X/B}^{[\otimes r]}\otimes \det(\sE)^{\otimes -1}\otimes \pi^*\sG^*)\neq 0$
and let
$s \in H^0(X,T_{X/B}^{[\otimes r]}\otimes \det(\sE)^{\otimes -1}\otimes \pi^*\sG^*)$
be a nonzero
section. Notice that, for any $\sigma \in \frS_r$ and any non negative integers $i$ and $j$ such that $r=2i+dj$,
$$\sigma\cdot [(S^{[2]} T_{X/B})^{[\otimes i]}
\otimes \det(T_{X/B})^{[\otimes j]}]
\otimes \det(\sE)^{\otimes -1}
\otimes\pi^*\sG^*$$
is a direct summand of
$$T_{X/B}^{[\otimes r]}\otimes \det(\sE)^{\otimes -1}\otimes \pi^*\sG^*.$$
By Lemma \ref{lemma:quadric_twisted_sections}, 
we may assume that
$$s\in H^0(X,(S^{[2]} T_{X/B})^{[\otimes i]}
\otimes \det(T_{X/B})^{[\otimes j]}
\otimes \det(\sE)^{\otimes -1}
\otimes\pi^*\sG^*)$$
and
$$s_{|X_{\bar\eta}}=q_{\bar\eta}^{\otimes i}\otimes {{\det}_{\bar\eta}}^{\otimes j}
\otimes g_{\bar \eta}$$
for some non negative integers $i$ and $j$ with $r=2i+dj$ and some
non zero section $g_{\bar\eta}\in \pi^* H^0(\bar\eta,\sG_{|\bar\eta})$. 
It follows that the induced map
$$
\sG \to \pi_*((S^{[2]} T_{X/B})^{[\otimes i]}
\otimes \det(T_{X/B})^{[\otimes j]}
\otimes \det(\sE)^{\otimes -1})
$$
has rank one and therefore, we may assume that $\sG$ is a line bundle (with $\deg(\sG) \ge 0$).
We obtain a map
$$\varphi_s:{\Omega^1_{X/B}}^{[\otimes i]}\to
{T_{X/B}}^{[\otimes i]}
\otimes \det(T_{X/B})^{[\otimes j]}
\otimes\det(\sE)^{\otimes -1}
\otimes\pi^*\sG^*$$
whose restriction to $X_{\bar\eta}$ is an isomorphism. Finally, we obtain a nonzero section

\begin{multline*}
s':=\det(\varphi_s)\in H^0(X,\det({T_{X/B}}^{[\otimes i]})\otimes
\det({T_{X/B}}^{[\otimes i]}
\otimes \det(T_{X/B})^{[\otimes j]}
\otimes\det(\sE)^{\otimes -1}
\otimes\pi^*\sG^*)) \\
\simeq H^0(X,\det(T_{X/B})^{[\otimes (2id^{i-1}+d^ij)]}
\otimes \det(\sE)^{\otimes -d^i}
\otimes\pi^*\sG^{\otimes -d^i}).
\end{multline*}
Observe that $s'$ does not vanish anywhere on a general fiber of $\pi$ and that any fiber of $\pi$ is integral.
Thus
$$-K_{X/B}\equiv \frac{d^i}{2id^{i-1}+d^ij} c_1(\det(\sE))+\pi^*\Delta$$
for some (integral) effective divisor 
$\Delta \ge \frac{d^i}{2id^{i-1}+d^ij} c_1(\sG)$ and
$-K_{X/B}$ is ample. But that contradicts Lemma \ref{lemma:-KX/Y_not_ample}.
\end{proof}

\begin{lemma}[{\cite[Theorem 3.1]{adk08}}]\label{lemma:-KX/Y_not_ample}
Let $X$ be a normal projective variety, $f:X\to C$ a surjective morphism onto a smooth
curve, and $\Delta\subseteq X$ a Weil
divisor such that $(X,\Delta)$ is log canonical over the generic
point of $C$. Then $-(K_{X/C}+\Delta)$ is not ample.
\end{lemma}

\begin{lemma}\label{lemma:uniruled_surface}
Let $S$ be a smooth projective surface equipped with a surjective morphism $\pi:S\to B$ 
with connected fibers
onto a smooth connected curve. Let $\sM$ be a nef and big line bundle on $S$.
Assume that, for a general point $b$ in $B$,
$\sM\cdot S_b=2r$ for some $r\ge 1$ and either
$g(B) \ge 1$ or $B=\bP^1$ and $S$ is a ruled surface over $B$. Then 
$h^0(S,T_S^{\otimes r}\otimes \sM^{\otimes -1})=0$.
\end{lemma}

\begin{proof}
Let $c:S \to \bar S$ be a minimal model. Write $\sM=c^*{\bar\sM}(-E)$ for some divisor
$E$ on $S$ supported on the exceptional locus of $c$. Observe that $E$ is effective and $\bar\sM$ is nef since $\sM$ is
nef.
Therefore, the natural map $T_S \to c^*T_{\bar S}$ induces 
an inclusion 
$H^0(S,T_S^{\otimes r}\otimes \sM^{\otimes -1})\subset 
H^0(\bar S,T_{\bar S}^{\otimes r}\otimes {\bar\sM}^{\otimes -1})$.
If $g(B)\ge 1$
then $\pi$ induces a morphism $\bar S\to B$ and
$\bar S$ is a ruled surface over $B$. We may thus assume that $S\to B$ is smooth.
Since $\sM\cdot S_b=2r$ and $T_{S/C}\cdot S_b=2$ for $b\in B$, we must have
$$H^0(S,T_{S/B}^{\otimes r}\otimes \sM^{\otimes -1})
=H^0(S,T_S^{\otimes r}\otimes \sM^{\otimes -1}).$$
Let us assume to the contrary that $h^0(S,T_S^{\otimes r}\otimes \sM^{\otimes -1})\neq 0$.
Then $r(-K_{S/B})\sim c_{1}(\sM)+\pi^*\Delta$ where $\Delta$ is an effective divisor on C and
$K_{S/B}$ is nef and big. But $K_{S/B}^2=0$ for any (geometrically) ruled surface, a contradiction.
\end{proof}

\subsection{Tools}The proof of the main Theorem will apply rational curves on $X$.
Our notation is consistent with that of \cite{kollar96}.

Let $X$ be a smooth complex projective uniruled variety and $H$ an
irreducible component of $\textup{RatCurves}(X)$.
Recall that only general points in
$H$ are in 1:1-correpondence with the associated curves in $X$. Let $\ell$
be a rational curve corresponding to a general point in $H$, with
normalization morphism $f:\bP^1\to \ell\subset X$. We denote by $[\ell]$ or
$[f]$ the point in $H$ corresponding to $\ell$.

We say that $H$ is a \emph{dominating family of rational curves on $X$}
if the corresponding universal family dominates $X$.  A dominating
family $H$ of rational curves on $X$ is called \emph{unsplit} if it is
proper.  It is called \emph{minimal} if, for a general point $x\in X$,
the subfamily of $H$ parametrizing curves through $x$ is proper.  

Let $H_1, \dots, H_k$ be minimal dominating families of rational curves on $X$.  For
each $i$, let $\overline H_i$ denote the closure of $H_i$ in $\textup{Chow}(X)$.  We define
the following equivalence relation on $X$, which we call $(H_1, \dots,
H_k)$-equivalence.  Two points $x,y\in X$ are $(H_1, \dots, H_k)$-equivalent if they
can be connected by a chain of 1-cycles from $\overline H_1\cup \cdots \cup \overline
H_k$.  By \cite{campana92} (see also \cite[IV.4.16]{kollar96}), there exists a proper
surjective morphism $\pi_0:X_0 \to Y_0$ from a dense open subset of $X$
onto a normal variety whose fibers are $(H_1, \dots, H_k)$-equivalence classes.  We
call this map the \emph{$(H_1, \dots, H_k)$-rationally connected quotient of $X$}.
For more details see \cite{kollar96}.

\begin{lemma}\label{lemma:extending_in_codim_1}
Let $X$ be a smooth complex projective variety and $H_1, \dots, H_k$ unsplit
dominating families of rational curves on $X$. Let $\pi_0:X_0 \to Y_0$ be the $(H_1, \dots, H_k)$-rationally
connected quotient of $X$.
If the geometric generic fiber is isomorphic to a projective space, then 
$\pi_0$ is a $\bP^d$-bundle in codimension one in $Y_0$ with $d:=\dim(X_0)-\dim(Y_0)$.
\end{lemma}

\begin{proof}By \cite[Lemma 2.2]{adk08}, we may assume that $\pi_0$ is
a proper surjective equidimensional morphism with integral fibers.
Let $C_0\subset Y_0$ be a general complete intersection curve. Set $X_{C_0}:=\pi_0^{-1}(C_0)$. Then
$X_{C_0}$ is a smooth variety. 
Let $\eta$ be the generic point of $C_0$ and
let $\sL_{C_0}$ be a line bundle on $X_{C_0}$ that restricts to 
$\sO_{\bP^d_{k(\eta)}}(1)$ on ${X_{C_0}}_{\eta}\simeq \bP^d_{k(\eta)}$ ($d \ge 1$)
(see the proof of
Lemma \ref{lemma:fibration_over_curve_projective_space}). Let $\sM$ be an ample
line bundle on
$X$ and
$r$ a positive integer such that $\sM_{|{X_{C_0}}_{\eta}}\simeq \sO_{\bP^d_{k(\eta)}}(r)$.

For each $i$, denote by
$H_i^j$, $1\leq j\leq n_i$, 
the unsplit covering families of
rational curves on $X_{C_0}$ whose general members correspond to rational 
curves on $X$ from the family $H_i$. Then 
$\pi_{C_0}:={\pi_0}_{|X_{C_0}}:X_{C_0} \to C_0$ is
the $(H_1^1,\ldots,H_1^{n_1},\ldots,H_k^1,\ldots,H_k^{n_k})$-rationally
connected quotient of $X_{C_0}$. 
Let $F$ be a fiber of $\pi_{C_0}$. Let
$[H_i^j]$ denote the class of a member of $H_i^j$ in $N_1(F)$ and
$\cH:=\{[H_i^j]\mid i=1,\dots,k, j=1,\dots,n_i\}$.  Then by
\cite[Proposition IV 3.13.3]{kollar96}, $N_1(F)$ is generated by $\cH$.
Therefore any curve contained
in any fiber of $\pi_{C_0}$ is numerically proportional in $N_1(X_{C_0}/C_0)$ to
a linear combination of the $[H_i^j]$'s.
Hence $N_1(X_{C_0}/C_0)$ is
generated by $\cH$ and 
$c_1(\sM_{X_{C_0}})= r c_1(\sL_{C_0}) \in N_1(X_{C_0}/C_0)$.
Thus $\sL_{X_{C_0}}$ is ample$/C_0$ and the claim follows from
\cite[Corollary 5.4]{fujita75}.
\end{proof}

\begin{notation}Let $X$ be a normal variety and $\sQ$ be a coherent torsion free sheaf of
$\sO_X$-modules. Say that a curve $C \subset X$ is a general complete intersection curve for $\sQ$ in the sense of
Mehta-Ramanathan if $C=H_1\cap\cdots\cap H_{\dim(X)-1}$, where $H_i \in |m_i H|$ are general, $H$ is an ample line
bundle on $X$ and the $m_i\in\bN$ are large enough so that the Harder-Narasimhan filtration of $\sQ$ commutes with
restriction to $C$.
\end{notation}

The following result was established in \cite[Proposition 4.1]{paris}.

\begin{lemma}\label{lemma:miyaoka}
Let $X$ and $Y$ be a smooth complex projective varieties with 
$\dim(Y) \ge 1$, $X_0$ an open subset of $X$ 
with $\codim_X(X\setminus X_0)\geq 2$, $Y_0$ a dense open subset of $Y$
and $\pi_0:X_0\to Y_0$
a proper surjective equidimensional morphism. 
Let $C\subset X_0$ be a general complete intersection curve for $\pi_0^*\Omega_{Y_0}^1$ in the sense of
Mehta-Ramanathan.
If $(\pi_0^*\Omega_{Y_0}^1)_{|C}$ is not nef then $Y$ is uniruled.
\end{lemma}

\begin{proof}Let us sketch the proof for the reader's convenience.
Fix an ample line bundle $H$ on $X$, and consider  general elements $H_i \in |m_i H|$, for 
$i\in\{1, \ldots, \dim(X)-1\}$, where the $m_i\in\bN$ are large enough so that the 
Harder-Narasimhan filtration of $\pi_0^* \Omega_{Y_0}^1$ commutes with restriction to 
$C:=H_1 \cap \cdots \cap H_{\dim(X)-1}$. Setting $Z:= H_1 \cap \cdots \cap H_{\dim(X)-\dim(Y)}$ and $Z_0:= Z \cap X_0$, we may assume that $Z$ is a smooth variety of 
dimension $\dim(Y)$, and that the restriction $\varphi_0 := \pi_0 |_{Z_0}$ is a finite morphism.  

By the hypothesis $(\varphi_0^*\Omega_{Y_0}^1)|_C$ is not nef, therefore 
$(\varphi_0^\ast T_{Y_0})|_C$ contains a subsheaf with positive slope.  Thus if we denote by 
$i: Z_0 \hookrightarrow Z$ the inclusion and by $\sF$ the reflexive sheaf 
$i_*(\varphi_0^* T_{Y_0})$, then the maximally destabilizing subsheaf $\sE$ of 
$\sF$ has positive slope (with respect to $H_{|Z}$).

Let $K$ be a splitting field of the function field $K(Z_0)$ over $K(Y_0)$, and let $\psi: T \to Z$ be 
the normalization of $Z$ in $K$. Consider $T_0:=\psi^{-1}(Z_0)$,  and let 
$j: T_0 \hookrightarrow T$ be the inclusion. If we denote by $\psi_0$ the restriction of $\psi$ to 
$T_0$, then the reflexive sheaf 
$\sF' := (\psi^* \sF)^{**} = j_*(\psi_0^* \varphi_0^* T_{Y_0})$ 
contains the sheaf $(\psi^* \sE)^{**}$. 
Notice that $(\psi^* \sE)^{**}$ has positive slope.
Consequently the 
maximally destabilizing subsheaf $\sE'$ of $\sF'$ has positive slope. Hence by 
replacing $Z_0$ with $T_0$, $\varphi_0$ with $\varphi_0 \circ \psi_0$, and 
$(\sF,\sE)$ with $(\sF',\sE')$ if necessary, we may assume 
that $K(Z_0) \supset K(Y_0)$ is a Galois extension with Galois group $G$.

Because of its uniqueness, the maximally destabilizing subsheaf $\sE$ of $\sF$ 
is invariant under the action of $G$. Thus by replacing $Z_0$ with another open subset of $Z$ if 
necessary, we may assume that there exists a saturated subsheaf $\sG$ of $T_{Y_0}$ 
such that $\sE=i_*( \varphi_0^*\sG)$.  

As $\sE$ has positive slope, it follows from 
\cite[Proposition 29 and Proposition 30]{kebekus_solaconde_toma07} that the 
vector bundles $\sE_{|C}$ and 
$(\sE\otimes \sE \otimes (\sF/\sE)^*)_{|C}$ are ample. 
The morphism $\varphi_0$ being finite, this implies that  $\sG_{|{\varphi_0(C)}}$ and 
$(\sG\otimes \sG \otimes (T_{Y_0}/\sG)^*)_{|{\varphi_0(C)}}$ are 
ample vector bundles too. In particular we deduce from this that 
$\textup{Hom}(\sG \otimes \sG, T_{Y_0}/\sG)=0 $, because the 
deformations of the curve $\varphi_0(C)$ dominate the variety $Y_0$. As a consequence 
$\sG$ is a foliation on $Y_0$.

Finally, by extending $\sG$ to a foliation $\widetilde{\sG}$ on the whole variety 
$Y$, we can conclude by using \cite[Theorem 1]{kebekus_solaconde_toma07}. Indeed it follows from the fact that 
$\widetilde{\sG}_{|{\varphi_0(C)}}$ is ample that the leaf of the foliation 
$\widetilde{\sG}$ passing through a general point of $\varphi_0(C)$ is rationally 
connected; in particular $Y$ is uniruled.
\end{proof}

The proof of our main result is based on the following result wich appears 
essentially in \cite{paris}.

\begin{cor}\label{corollary:miyaoka}
Let $X$ be a smooth complex projective variety, $X_0$ an open subset of $X$ 
with $\codim_X(X\setminus X_0)\geq 2$, $Y_0$ a smooth variety
with $\dim(Y_0)\ge 1$
and $\pi_0:X_0\to Y_0$
a proper surjective equidimensional morphism. 
Assume that the generic fiber of $\pi_0$ is isomorphic to a projective space.
Let $C$ be a general complete intersection curve for $\pi_0^*\Omega_{Y_0}^1$ in the sense of
Mehta-Ramanathan.
If $(\pi_0^*\Omega_{Y_0}^1)_{|C}$ is not nef then
there exists a minimal free morphism $f:\bP^1\to Y_0$.
\end{cor}

\begin{proof}
Let $Y$ be a smooth projective variety containing $Y_0$ as a dense open subset.
By Lemma \ref{lemma:miyaoka}, $Y$ is uniruled. Let $H_Y$ be a minimal dominating family of rational curves on $Y$. Since the generic fiber of $\pi_0$ is isomorphic to a projective space, there
exists a dominating family $H_X$ of rational curves on $X$ such that for a general member $[f]\in H_X$, $[\pi_0\circ f]$ is a general member of $H_Y$. By \cite[Proposition II 3.7]{kollar96}, if
$[f]\in H_X$ is a general member then $f(\bP^1)\subset X_0$. The claim follows from
\cite[Corollary IV 2.9]{kollar96}.
\end{proof}

The following Lemma is certainly well known to experts.
We include a proof for lack of an adequate reference.

\begin{lemma}\label{lemma:morphism}
Let $X$ be a smooth complex variety and $H$ a minimal dominating family of rational 
curves on $X$. Let $x$ be a general point in $X$ and $[\ell]\in H$ with $x\in\ell$. If 
$T_{\ell,x}$ does not depend on $\ell \ni x$ then 
there exists a non empty open subset $X_0$ in $X$ and a proper surjective morphism 
$\pi_0:X_0\to Y_0$ onto a variety $Y_0$ such that any fiber 
of $\pi_0$ is a rational curve from the family $H$.
\end{lemma}

\begin{proof}Let $[f]\in H$ be a general member.
By \cite[Corollary IV 2.9]{kollar96},
$f^*T_X\simeq \sO_{\bP^1}(2)\oplus \sO_{\bP^1}(1)^{\oplus d}
\oplus\sO_{\bP^1}^{\oplus (n-d-1)}$
with $d:=-K_X\cdot f_*\bP^1-2$. 
Let $x$ be a general point in $X$ with $x\in\ell:=f(\bP^1)$.
By
\cite[Proposition 2.3]{hwang}, $d=0$ using the fact that 
$T_{\ell,x}$ does not depend on $\ell \ni x$. 

Let $\bar H$ be the normalization of the closure of 
$H$ in $\textup{Chow}(X)$ and $\bar U$ the normalization of the universal family. Let us denote by
$\bar\pi:\bar U\to \bar H$ and $\bar e:\bar U\to X$ the universal morphisms.
By shrinking $H$ if necessary, we may assume that $H$ parametrizes free
morphisms. Then $H$ is smooth (see \cite[Theorem I 2.16]{kollar96}) and 
$e:=\bar e_{|U}:U\to X$ is \'etale
where $U:=\bar\pi^{-1}(H)$ (see \cite[Proposition II 3.4]{kollar96}). 

It remains to show that there exists a dense open subset $H_0$ of $H$ such
that the restriction of $\bar e$ to $\bar\pi^{-1}(H_0)$ induces an isomorphism
onto 
the open set $\bar e(\bar\pi^{-1}(H_0)$.
By Zariski's main Theorem, it is enough to prove that $\bar e$ is
birational. 
We argue by contradiction. Then there exists a curve $C\subset \bar U$ such that
$\dim(\bar\pi(C))=1$ and 
$\bar e(C)=\ell$. Let $c$ be a general point in $C$. Then 
$d_c\bar e(T_{C,c})=d_c\bar e(T_{\bar\pi^{-1}(\bar\pi(c),c})=T_{\ell,\bar e(c)}$.
But that contradicts the fact that $\bar e$ is \'etale at $c$.
The claim follows.
\end{proof}

\subsection{Characterizations of projective spaces and hyperquadrics}

The proof of the main Theorem stated in the introduction is based on the following result whose
proof is similar to that of \cite[Theorem 6.3]{adk08}.

\begin{prop}\label{proposition:picard_number_one}
Let $X$ be a smooth complex projective $n$-dimensional variety with
$\rho(X)=1$ and
$\sE$ an ample vector bundle on $X$ of rank $r+k$ with $r\ge 1$ and $k\ge 0$. If 
$h^0(X,T_{X}^{\otimes r}\otimes \det(\sE)^{\otimes -1})\neq 0$, then
either $X\simeq \bP^{n}$, or 
$k=0$ and $X\simeq Q_{n}$ ($n\neq 2$).
\end{prop}

\begin{proof}Let us give the proof following \cite{adk08}.
First notice that $X$ is uniruled by \cite{miyaoka87}, and hence a
Fano manifold with $\rho(X)=1$. The result is clear if $\dim X=1$,
so we assume that $n\geq 2$.  Fix a minimal dominating family $H$ of
rational curves on $X$. Let $\sL$ be an ample line bundle on $X$ such that
$\textup{Pic}(X)=\bZ[\sL]$.
  
Let $\sE'\subset T_X$ be the maximally destabilizing subsheaf of $T_X$; $\sE'$ is a reflexive sheaf of 
rank $r' \ge 1$. By \cite[Lemma 6.2]{adk08}, $\mu_{_\sL}(\sE') \geq
\frac{\mu_{_\sL}(\det(\sE))}{r}.$ 
Notice that $\mu_{_\sL}(\det(\sE)) \ge r+k$ since $\sE$ is ample.
This implies that
$\frac{\deg(f^*\sE')}{r'}\geq \frac{\deg(f^*\det(\sE))}{r}\ge \frac{r+k}{r}\ge 1$ for a general member
$[f]\in H$.  If $r'=1$, then $\sE'$ is ample and we are done
by Wahl's Theorem. 
If $f^*\sE'$ is ample, then $X\simeq \bP^n$ by \cite[Proposition 2.7]{adk08}, using the fact that
$\rho(X)=1$.

Otherwise, as $f^*\sE'$ is a subsheaf of
$f^*T_X\simeq \sO_{\bP^1}(2)\oplus \sO_{\bP^1}(1)^{\oplus d}\oplus
\sO_{\bP^1}^{\oplus (n-d-1)}$ (see \cite[Corollary IV 2.9]{kollar96}), we must have 
$\deg(f^*\det(\sE'))=r'$, $\deg(f^*\det(\sE))=r$, $k=0$ and
$f^*\sE'\simeq \sO_{\bP^1}(2)\oplus
\sO_{\bP^1}(1)^{\oplus r'-2}\oplus \sO_{\bP^1}$ for a general $[f]\in H$.
  Then $\sO_{\bP^1}(2)\subset
f^*\sE'$ for general $[f]\in H$.
Thus by
\cite[Proposition 2.3]{hwang}, 
$(f^*T_X^+)_p
\subset (f^*\sE')_p$ for a general $p\in \bP^1$ and a
general $[f]\in H$.  Since $f^*\sE'$ is a subbundle of $f^*T_X$, we
have an inclusion of sheaves $f^*T_X^+\into f^*\sE'$, and thus
$f^*\det(\sE')=f^*\omega_X^{-1}$.  Since $\rho(X)=1$, this implies
that $\det \sE'=\omega_X^{-1}$, and thus $0\neq
h^{0}(X,\wedge^{r'} T_{X}\otimes \omega_{X})=h^{n-r'}(X,\sO_{X})$.
The latter is zero unless $r'=n$ since $X$ is a Fano manifold. 
Notice that $\deg(f^*\det(\sE))=r$. It follows that, for any $[f]\in H$,
$f^*\sE\simeq\sO_{\bP^1}(1)^{\oplus r}$. By \cite[Proposition 1.2]{andreatta_wisniewski}
(see also \cite[Theorem 4.3]{ross}),
$\sE\simeq \sL^{\oplus r}$ and $\deg(f^*\sL)=1$.
If $n=r'$, then we must have $\omega_X^{-1}\simeq \det(\sE')\simeq \sL^{\otimes n}$.
Hence $X\simeq Q_n$ by \cite{kobayashi_ochiai73}.
\end{proof}

We will need the following auxiliary result.

\begin{lemma}\label{lemma:unsplit}
Let $X$ be a smooth complex projective variety 
and $\sE$ an ample vector bundle on $X$ of rank $r+k$ with $r\ge 2$ and $k\ge 0$. 
Assume that $X$ is uniruled and
fix a minimal dominating family $H$ of
rational curves on $X$.
If $h^0(X,T_{X}^{\otimes r}\otimes \det(\sE)^{\otimes -1})\neq 0$, then
$H$ is unsplit.
\end{lemma}

\begin{proof}The proof is similar to that of \cite[Proposition 4.2]{paris}.
Let $[f] \in H$ be a general member.
Let us assume to the contrary that 
$h^0(X,T_{X}^{\otimes r}\otimes 
\det(\sE)^{\otimes -1})\neq 0$ and 
$f_*(\bP^1)\equiv C_1+C_2$ with $C_1$ and $C_2$ nonzero integral effective rational
$1$-cycles.
Notice first that $\det(\sE)\cdot C\ge r+k$ for all rational curve $C\subset X$.
By \cite[Corollary IV 2.9]{kollar96},
$f^*T_X\simeq \sO_{\bP^1}(2)\oplus \sO_{\bP^1}(1)^{\oplus d}
\oplus\sO_{\bP^1}^{\oplus (n-d-1)}$ and we must have 
$\deg(f^*\det(\sE)) \le 2r$. Finally, $2(r+k)\le \deg(f^*\det(\sE)) \le 2r$ and we must have $k=0$,
$\deg(f^*\det(\sE)) = 2r$ and 
$f^*\det(\sE)\simeq \sO_{\bP^1}(2r)\subset 
f^*\wedge^r(T_X)\simeq 
\wedge^r(\sO_{\bP^1}(2)\oplus \sO_{\bP^1}(1)^{\oplus d}
\oplus\sO_{\bP^1}^{\oplus (n-d-1)})$. 
Hence $T_{\ell,x}^{\otimes r}=\det(\sE)_x\subset T_{X,x}^{\otimes r}$ for
a general point $x$ in $\ell$ and 
therefore, $T_{\ell,x}$
does not depend
on $\ell \ni x$. Thus, by Lemma \ref{lemma:morphism},
there exists a non empty open subset $X_0$ in $X$ and a proper surjective morphism 
$\pi_0:X_0\to Y_0$ onto a variety $Y_0$ such that any fiber 
of $\pi_0$ is a rational curve from the family $H$ and 
$\det(\sE)_{|X_0}\simeq T_{X_0/Y_0}^{\otimes r}$. Let $\sL\subset T_X$ be the saturated line bundle such that $T_{X_0/Y_0}\simeq \sL_{|X_0}$. Notice that
$\det(\sE)\subset \sL^{\otimes r}$ with equality on $X_0$.
Let $C\subset X$ be a general complete intersection curve and let $S$ be the normalization of the closure in $X$ of 
$\pi_0^{-1}(\pi_0(C\cap X_0))$. By \cite[Lemme 1.2]{druel04} (or \cite[Proposition 4.5]{adk08}),
the map $\Omega_X^1 \to \sL^{\otimes -1}$ induces a map 
$\Omega_S^1\to {\sL_S}^{\otimes -1}$ where $\sL_S$ denotes the pull-back of $\sL$ to $S$.
Notice that $\pi_0$ induces a surjective morphism $\pi_S:S\to B$ onto a smooth curve. 
By Lemma \ref{lemma:uniruled_surface}, $\dim(X_0)\neq 2$. Thus, we
may assume $g(B) \ge 1$.
Let $\tilde S \to S$ be a minimal desingularization of $S$. By \cite[Proposition 1.2]{burns_wahl74},
$\Omega_S^1\to {\sL_S}^{\otimes -1}$ extends to
$\Omega_{\tilde S}^1\to {\sL_{\tilde S}}^{\otimes -1}$. Let $\pi_{\tilde S}:\tilde S\to B$ be the induced morphism. 
By replacing $\sL_{\tilde S}$ with its saturation in $T_{\tilde S}$, we may assume
$\det(\sE)_{\tilde S}\subset \sL_{\tilde S}^{\otimes r}\subset T_{\tilde S}^{\otimes r}$.
Observe also that, for a general point $b$ in $B$, 
$\det(\sE)_{\tilde S}\cdot \tilde S_b=2r$. But that contradicts Lemma \ref{lemma:uniruled_surface}.
\end{proof}

Now we can prove our main theorems.

\begin{thm}\label{theorem:main}
Let $X$ be a smooth complex projective variety 
and $\sE$ an ample vector bundle on $X$ of rank $r+k$ with $r\ge 1$ and $k\ge 0$ and such that
$h^0(X,T_{X}^{\otimes r}\otimes \det(\sE)^{\otimes -1})\neq 0$.
\begin{enumerate}
\item If $k\ge 1$ then $X\simeq \bP^{n}$.
\item If $k=0$ then either $X\simeq \bP^{n}$, or 
$X\simeq Q_{n}$.
\end{enumerate}
\end{thm}

\begin{proof}
We shall proceed by induction on $n:=\dim(X)$.
The result is clear if $n=1$, so we assume that $n\geq 2$.
If $r+k=1$ then we are done by Wahl's Theorem so we assume that $r+k\ge 2$.
Notice that $X$ is uniruled by \cite{miyaoka87}.
Fix a minimal dominating family $H$ of
rational curves on $X$. By Lemma \ref{lemma:unsplit}, $H$ is unsplit.
Let $\pi_0:X_0\to Y_0$ be the $H$-rationally connected quotient of $X$.
By \cite[Lemma 2.2]{adk08}, we may assume 
$\codim_X(X\setminus X_0)\ge 2$ and $\pi_0$ is an equidimensional surjective morphism
with integral fibers. By shrinking $Y_0$ if necessary, we may also assume that $Y_0$ is 
smooth.

By Proposition \ref{proposition:picard_number_one}, we may assume
$\rho(X) \ge 2$. By \cite[Proposition IV 3.13.3]{kollar96}, we must have $\dim(Y_0)\ge 1$. 

Let $F$ be a general fiber of $\pi_0$.
There exist (see \cite[Lemma 5.1]{adk08} or \cite[Lemme 2.1]{paris}) non negative 
integers $i$ and $j$ with $i+j=r$ such that 
$h^0(X,T_{X_0/Y_0}^{[\otimes i]}\otimes \det(\sE)_{|X_0}^{\otimes -1}
\otimes \pi_0^* T_{Y_0}^{\otimes j}) \neq 0$ and
$h^0(F,T_F^{\otimes i}\otimes \det(\sE)_{|F}^{\otimes -1})\neq 0$. Notice that $i\ge 1$ since
$\det(\sE)_{|F}$ is an ample line bundle and $d:=\dim(F)\ge 1$. The induction hypothesis implies
that $F\simeq \bP^d$ if $i<r$ or $k\ge 1$ and 
either $F\simeq \bP^d$ or $F\simeq Q_d$
if $i=r$ and $k=0$.

Let $C\subset X_0$ be
a general complete intersection curve (with respect to some very ample line bundle on $X$).
Let $X_C$ be the normalization of $\pi_0^{-1}(\pi_0(C))$.
Let $\pi_C:X_C \to C$ be the induced map.
Notice that $X_C$ is the normalization of $C \times _{Y_0} X_0$ and that
$C \times _{Y_0} X_0$ is regular in codimension one. Hence, we must have
$h^0(X_C,T_{X_C/C}^{[\otimes i]}\otimes \det(\sE)_{|X_C}^{\otimes -1}
\otimes \pi_C^*( {\Omega_{Y_0}^1}_{|C}^{\otimes -j})) \neq 0$.
Let us assume that either $(\pi_0^*\Omega_{Y_0}^1)_{|C}$ is a nef vector bundle or $i=r$.
If the geometric generic fiber of $\pi_0$ is isomorphic to a projective space then 
$\pi_0$ is a $\bP^d$-bundle by Lemma
\ref{lemma:extending_in_codim_1}. But that contradicts
Lemma \ref{lemma:fibration_over_curve_projective_space}. Thus
the geometric generic fiber of $\pi_0$ is isomorphic to a (smooth) hyperquadric.
But that contradicts Proposition \ref{proposition:fibration_over_curve_quadric}.

Thus $i<r$, $F\simeq \bP^d$ and
by Lemma \ref{corollary:miyaoka},
there exists a minimal free morphism $f:\bP^1\to Y_0$. By generic smoothness, we may assume that
$X_f:=\bP^1\times_{Y_0} X_0$ is smooth. We may also assume that
$h^0(X_f,T_{X_f/\bP^1}^{[\otimes i]}\otimes \det(\sE)_{|X_f}^{\otimes -1}
\otimes \pi_f^*({T_{Y_0}}_{|\bP^1}^{\otimes j})) \neq 0$.
Let $\sL_f$ be a line bundle on $X_f$ that restricts to $\sO_{\bP^d}(1)$ on $F\simeq \bP^d$ (see the proof of
Lemma \ref{lemma:fibration_over_curve_projective_space}).
By \cite[Corollary 5.4]{fujita75}, $\pi_f:X_f\to \bP^1$ is a $\bP^d$ bundle.
It follows from Lemma \ref{lemma:bundle_over_line} that $k=0$, $d=1$, 
$(X_f/\bP^1)\simeq (\bP^1\times \bP^1/\bP^1)$ and 
$\det(\sE)_{|X_f}\simeq\sO_{\bP^1}(2)\boxtimes\sO_{\bP^1}(2)$.
Since $\sE$ is ample, $X$ admits an unsplit dominating covering family $H'$ of rational curves whose general member
corresponds to a ruling of $X_f$ that is not contracted by $\pi$.
Let $\pi_1:X_1\to Y_1$ be the $(H,H')$-rationally connected quotient of $X$. 
By \cite[Lemma 2.2]{adk08}, we may assume 
$\codim_X(X\setminus X_1)\ge 2$ and $\pi_1$ is an equidimensional surjective morphism
with integral fibers. By shrinking $Y_1$ if necessary, we may also assume that $Y_1$ is 
smooth. Replacing $\pi_0:X_0\to Y_0$ with $\pi_1:X_1\to Y_1$ above, we obtain a contradiction unless $X\simeq \bP^1\times \bP^1$.
\end{proof}

\begin{proof}[Proof of Theorem A]
By Theorem \ref{theorem:main},
$X\simeq \bP^{n}$
and by Lemma \ref{lemma:twisted_sections_projective_space},
$\det(\sE)\simeq \sO_{\bP^{n}}(l))$ with
$r+k\le l\le \frac{r(n+1)}{n}$.
\end{proof}

\begin{proof}[Proof of Theorem B]
By Theorem \ref{theorem:main},
either $X\simeq \bP^{n}$ or $X\simeq Q_{n}$. If $X\simeq \bP^{n}$, then the claim follows from 
Lemma \ref{lemma:twisted_sections_projective_space}. Let us assume 
$X\simeq Q_{n}$. By Lemma \ref{lemma:quadric_twisted_sections},
$\det(\sE)\simeq \sO_{Q_n}(r)$. Thus, for any line $\bP^1\subset Q_n\subset \bP^{n+1}$,
$\sE_{|\bP^1}\simeq \sO_{\bP^1}(1)^{\oplus r}$, and the claim follows
from \cite[Proposition 1.2]{andreatta_wisniewski} (see also \cite[Theorem 4.3]{ross}).
\end{proof}

\providecommand{\bysame}{\leavevmode\hbox to3em{\hrulefill}\thinspace}
\providecommand{\MR}{\relax\ifhmode\unskip\space\fi MR }
\providecommand{\MRhref}[2]{%
  \href{http://www.ams.org/mathscinet-getitem?mr=#1}{#2}
}
\providecommand{\href}[2]{#2}

%\bibliographystyle{amsalpha}
%\bibliography{epq2}
\end{document}